\documentclass[oneside,english]{amsart}
\usepackage[T1]{fontenc}
\usepackage[latin9]{inputenc}
\usepackage{geometry}
\usepackage{amsthm}
\usepackage{amssymb}
\usepackage{graphicx}
\usepackage{setspace}
\makeatletter
\numberwithin{equation}{section}
\numberwithin{figure}{section}
\newenvironment{lyxlist}[1]
{\begin{list}{}
{\settowidth{\labelwidth}{#1}
 \setlength{\leftmargin}{\labelwidth}
 \addtolength{\leftmargin}{\labelsep}
 }}
{\end{list}}
\theoremstyle{plain}
\newtheorem{thm}{\protect\theoremname}
  \theoremstyle{plain}
  \newtheorem{lem}[thm]{\protect\lemmaname}
  \theoremstyle{plain}
  \newtheorem{cor}[thm]{\protect\corollaryname}
 \theoremstyle{definition}
 \newtheorem*{defn*}{\protect\definitionname}

\makeatother

\usepackage{babel}
  \providecommand{\corollaryname}{Corollary}
  \providecommand{\definitionname}{Definition}
  \providecommand{\lemmaname}{Lemma}
\providecommand{\theoremname}{Theorem}

\begin{document}

\title[The Perpendicular Bisector Construction in $n$-Dimensional geometry]{The Perpendicular Bisector Construction in $n$-Dimensional Euclidean
and Non-Euclidean Geometries}

\author{Emmanuel Tsukerman}

\date{\today}
\begin{abstract}
The {}``Perpendicular Bisectors Construction'' is a natural way
to seek a replacement for the circumcenter of a noncyclic quadrilateral
in the plane. In this paper, we generalize this iterative construction
to a construction on polytopes with $n$ vertices in $(n-2)$-dimensional
Euclidean, Hyperbolic and Elliptic geometries. We then show that a
number of nice properties concerning this iterative construction continue
to hold in these geometries. We also introduce an analogue of the
isoptic point of a quadrilateral, which is the limit point of the
Perpendicular Bisectors Construction, in $\mathbb{R}^{n}$ and prove
some of its properties.
\end{abstract}
\maketitle

\section{Background}

A natural way to seek a replacement for the circumcenter of a cyclic
planar quadrilateral in the case when the quadrilateral is noncyclic
is to proceed with the following iterative construction: 
\begin{itemize}
\item For every $3$ vertices of a quadrilateral $Q^{(1)}$, determine the
circumcenter. The resulting 4 points form a new quadrilateral $Q^{(2)}$.
The construction can then be iterated on $Q^{(2)}$ and then on $Q^{(3)}$,
etc.
\end{itemize}
This construction is known as the {}``Perpendicular Bisectors Construction''
since the sides of $Q^{(i+1)}$ are determined using the perpendicular
bisectors of the sides of $Q^{(i)}$. 

The construction is so natural that it was looked at before a number
of times. In particular, the following problem about the Perpendicular
Bisectors Construction was proposed by Josef Langr \cite{Langr} in
1953:

\emph{The perpendicular bisectors of the sides of a quadrilateral
$ABCD$ form a quadrilateral $A_{1}B_{1}C_{1}D_{1}$ and the perpendicular
bisectors of the sides of $A_{1}B_{1}C_{1}D_{1}$ form a quadrilateral
$A_{2}B_{2}C_{2}D_{2}$. Show that $A_{2}B_{2}C_{2}D_{2}$ is similar
to $ABCD$ and find the ratio of similitude.}

Given that the problem is relatively simple, it is surprising that
no solutions were published in English for over half a century. The
problem was mentioned by C.S. Ogilvy (\cite{Ogilvy}, p. 80) as an
example of an unsolved problem. According to an article on Alexander
Bogomolny's \emph{Cut-the-knot }website \cite{cut-the-knot}, \emph{{}``B.
Grünbaum \cite{Grunbaum} wrote about the problem in 1993 as an example
of an unproven problem whose correctness could not be doubted... {[}D.
Schattschneider{]} proved several particular cases of the problem,
but the general problem remained yet unsolved. It looks like, by that
time, the problem made it into the mathematical folklore. It reached
Dan Bennett by the word of mouth and its simplicity had piqued his
interest. He published a solution \cite{D. Bennet} in 1997 to a major
part of the problem under an additional assumption that was promptly
removed by J. King \cite{King} who (independently) also supplied
a proof based on the same ideas''}. A paper by G.C. Shepard \cite{Shepard}
also found an expression for the ratio, and several simpler forms
of the expression are given by Radko and Tsukerman in \cite{RadkoTsukerman}.

In the same paper, Radko and Tsukerman show that the construction
(or, if $ABCD$ is non-convex, the reverse construction) has a limit
called the \textit{isoptic point}, due to its property of {}``being
seen'' at equal angles from each of the triad circles of the quadrilateral.
This point has many beautiful properties, such as having a parallelogram
pedal, being the unique intersection of the $6$ circles of similitude
of a quadrilateral and having many of the properties expected of a
replacement of the circumcenter.

\section{Main Results}

We introduce a generalization to the Perpendicular Bisectors Construction,
which we apply to polytopes with $n$ vertices in $(n-2)$-dimensional
Euclidean, Hyperbolic and Elliptic geometries. We prove the remarkable
property that for any dimension and any geometry previously mentioned,
the $i$th generation polyope $P^{(i)}$ and $(i+2)$th generation
polytope $P^{(i+2)}$ are in perspective for each $i$. After showing
how the iterative construction in any of the geometries can be reversed
via isogonal conjugation, we show that in the case of Euclidean geometry,
all $P^{(2k)}$ are homothetic and all $P^{(2k+1)}$ are homothetic,
and the center of homothecy is the same for both families of polytopes.
Finally, we define an analogue of the isoptic point in $\mathbb{R}^{n}$
and prove some of its properties.

\section{Preliminaries and Notation}

We consider $d$-dimensional Euclidean, Hyperbolic or Elliptic space,
where $d=n-2$. Recall that a hyperplane is a $(d-1)$-flat and that
the mediator hyperplane of a segment $P_{1}P_{2}$, denoted $PB(P_{1}P_{2})$
throughout, is the hyperplane passing through the midpoint of $P_{1}P_{2}$
orthogonal to that segment. By a hypersphere, we will specifically
mean a $(d-1)$-sphere. A facet of a polytope is a face with affine
dimension $d-1$.

Our approach to proving the perspectivity $P^{(i)}$ and $P^{(i+2)}$
will naturally involve projective geometry. Specifically, we will
view Euclidean, Hyperbolic and Elliptic geometries as embedded inside
of real projective $d$-space. For the convenience of the reader,
we now give a brief overview of how to do so.

Recall that a correlation in real projective $n$-space $\mathbb{R}P^{n}$
is a one-to-one linear transformation taking points into hyperplanes
and vice verse. A \textit{polarity} is an involutory correlation,
and we call the image of a point $P$ under a polarity its \textit{polar},
and the image of a hyperplane $q$ its \textit{pole}. We shall utilize
the following facts:
\begin{enumerate}
\item A point $P$ is on the polar of a point $Q$ under a given polarity
if and only if $Q$ is on the polar of $P$ under this same polarity.
Similarly, a hyperplane $p$ is incident to the pole of hyperplane
$q$ if and only if $q$ is incident to the pole of $p$. We call
such $P$ and $Q$ \textit{conjugate points }and such $p$ and $q$
\textit{conjugate hyperplanes}. A point that lies on its own polar
is called a\textit{ self-conjugate point}. Similarly, a hyperplane
incident with its own pole is a \textit{self-conjugate hyperplane}.
\item A nonempty set of self-conjugate points with respect to a given polarity
is a quadric and any quadric is a set of self-conjugate points with
respect to some polarity. 
\end{enumerate}
As an illustration, given a point $P$ outside of a conic in the projective
plane, there are two tangents passing through $P$. The polar of $P$
is the line incident to the two points of tangency.

\begin{figure}[h]
\includegraphics[scale=0.3]{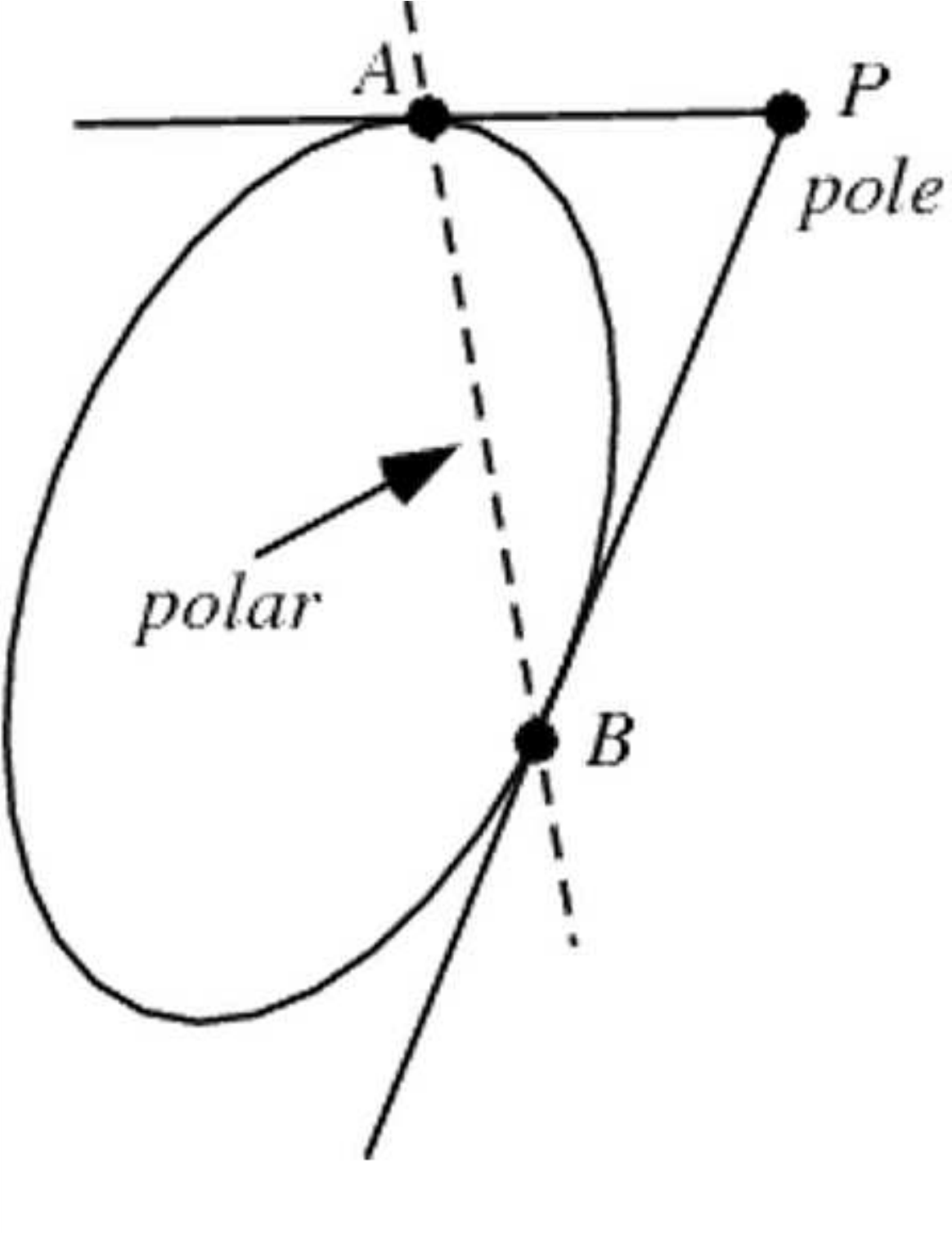}

\caption{Pole-polar relation in the projective plane \cite{Mathworld}.}

\end{figure}

To obtain Euclidean, Hyperbolic and Elliptic geometries as subgeometries
of projective geometry we fix a polarity and an associated quadric
$\Gamma$, depending on the geometry. We make the following identifications
in the projective plane, and the more general identification for projective
$n$-space are similar. 
\begin{enumerate}
\item For Hyperbolic geometry, the points inside of the quadric are the
the (ordinary) points of the geometry, points on the quadric are \textit{ideal
points }and points outside of the quadric are \textit{hyperideal points}.
Hyperbolic lines are the parts of the projective lines having ordinary
points. Two hyperbolic lines are parallel (ultraparallel) if the corresponding
projective lines intersect in ideal (ultraideal) points. They are
perpendicular if they are conjugate with respect to $\Gamma$.
\item For Elliptic geometry, the ordinary points are the points of the projective
plane and the lines are the lines of the projective plane. Two elliptic
lines are perpendicular if the corresponding projective lines are
conjugate with respect to $\Gamma$.
\item In Euclidean geometry, the ordinary points are the points of the projective
plane not on $\Gamma$ and the ideal points are the points on $\Gamma$.
Two lines are perpendicular if their ideal points correspond under
the absolute projectivity. 
\end{enumerate}
We refer the reader to \cite{Cederberg} for a more comprehensive
discussion on the subgeometries of projective space, e.g., on the
defiinition of angles, distances, etc. 

$ $

We list here some of the notation which will be employed throughout:
\begin{lyxlist}{00.00.0000}
\item [{$N^{(i)}$}] \noindent is the $i$th generation set of vertices
constructed via the iterative process.
\item [{$(P_{1}\cdot\cdot\cdot P_{n-1})$}] \noindent will denote the unique
hypersphere through points $P_{1},...,P_{n-1}$.
\item [{$PB(P_{1}P_{2})$}] \noindent will denote the mediator of line
segment $P_{1}P_{2}$.
\item [{$PB(H_{1},H_{2})$}] \noindent will denote a common perpendicular
of hyperplanes $H_{1}$ and $H_{2}$. In Euclidean geometry, this
will simply mean that $PB(H_{1},H_{2})$ is perpendicular to both
$H_{1}$ and $H_{2}$.
\item [{$Iso_{P_{1}\cdot\cdot\cdot P_{n-1}}P_{n}$}] \noindent denotes
the isogonal conjugate of the point $P_{n}$ in the simplex $P_{1}\cdot\cdot\cdot P_{n-1}$.
\item [{$P^{(i)}\sim P^{(j)}$}] \noindent denotes that a polytope $P^{(i)}$
with vertices $N^{(i)}$ can be chosen to have the same combinatorial
type as a polytope $P^{(j)}$ with vertices $N^{(j)}$, and $P^{(i)}$
and $P^{(j)}$ are similar.
\item [{$|P^{(i)}|$}] \noindent denotes the volume of $P^{(i)}$.
\end{lyxlist}

\section{The Generalized Iterative Process}

Consider a set $N^{(1)}$ of $n$ points $V_{1},V_{2},...,V_{n}$
in $(n-2)$-dimensional space $V$. For convenience, we will say that
$V_{i}=V_{n+i}$ for each $i$. When $V$ is Euclidean geometry, we
will require that any $n-1$ be affinely independent and in Hyperbolic
geometry, we will also require that any $n-1$ can be circumscribed
in a hypersphere, i.e. the circumcenter of the hypersphere is an ordinary,
rather than ideal or hyperideal, point. Our generalization of the
iterative process is as follows. 
\begin{itemize}
\item For each vertex $V_{i}$, $i=1,...,n$, construct the center $V_{i}^{(2)}$
of a hypersphere $(V_{i+1}\cdot\cdot\cdot V_{i+n-1})$. 
\end{itemize}
The vertices $V_{i}^{(2)}$, $i=1,...,n$, determine a new set of
$n$ points, which we will denote by $N^{(2)}$. 
\begin{itemize}
\item The construction is then repeated on $N^{(2)}$ to produce $N^{(3)}$,
etc.
\end{itemize}
It is easy to see that $N^{(2)}$ degenerates to a single point if
and only if the points of $N^{(1)}$ are conhyperspherical, meaning
that they can be inscribed in a hypersphere.

Moreover,
\begin{lem}
\label{lem:affinely dependent}In $(n-2)$-dimensional Euclidean geometry,
the set $N^{(2)}$ contains a point at infinity (an ideal point) if
and only if some $n-1$ points of $N^{(1)}$ are affinely dependent.\end{lem}
\begin{proof}
Assume that some $n-1$ points $V_{i}=(x_{i1},...,x_{id}),$ $i=1,...,n-1$,
of $N^{(1)}$ are affinely dependent. The equation of the hypersphere
passing through such $n-1$ points is given by
\[
\left|\left(\begin{array}{cccccccc}
x_{1}^{2}+x_{2}^{2}+...+x_{d}^{2} & x_{1} & x_{2} & \cdot & \cdot & \cdot & x_{d} & 1\\
x_{11}^{2}+x_{12}^{2}+...+x_{1d}^{2} & x_{11} & x_{12} & \cdot & \cdot & \cdot & x_{1d} & 1\\
x_{21}^{2}+x_{22}^{2}+...+x_{2d}^{2} & x_{21} & x_{22} & \cdot & \cdot & \cdot & x_{2d} & 1\\
\cdot & \cdot & \cdot & \cdot &  &  & \cdot & \cdot\\
\cdot & \cdot & \cdot &  & \cdot &  & \cdot & \cdot\\
\cdot & \cdot & \cdot &  &  & \cdot & \cdot & \cdot\\
x_{n-1,1}^{2}+x_{n-1,2}^{2}+...+x_{n-1,d}^{2} & x_{n-1,1} & x_{n-1,2} & \cdot & \cdot & \cdot & x_{n-1,d} & 1
\end{array}\right)\right|=0
\]

By expanding minors across the first row, we can find the coefficient
of the quadratic terms of the hypersphere to be 
\[
\left|\left(\begin{array}{ccccccc}
x_{11} & x_{12} & \cdot & \cdot & \cdot & x_{1d} & 1\\
x_{21} & x_{22} & \cdot & \cdot & \cdot & x_{2d} & 1\\
\cdot & \cdot & \cdot &  &  & \cdot & \cdot\\
\cdot & \cdot &  & \cdot &  & \cdot & \cdot\\
\cdot & \cdot &  &  & \cdot & \cdot & \cdot\\
x_{n-1,1} & x_{n-1,2} & \cdot & \cdot & \cdot & x_{n-1,d} & 1
\end{array}\right)\right|,
\]

which is zero, so the center of the hypersphere is an ideal point.

Conversely, if the center of the hypersphere is ideal, the coefficient
of the quadratic terms must be zero.
\end{proof}
We will call $N^{(k)}$ degenerate if it contains $n-1$ points which
are are affinely dependent. 
\begin{lem}
In $(n-2)$-dimensional Euclidean geometry, if the set $N^{(1)}$
is nondegenerate and is not conhyperspherical, then any $n-1$ points
of $N^{(2)}$ are affinely independent.\end{lem}
\begin{proof}
For convenience, we denote the hypersphere $(V_{i+1}\cdot\cdot\cdot V_{n+i-1})$
by $(V_{i}^{(2)})$. Assume by contradiction that some $n-1$ points
$V_{1}^{(2)},V_{2}^{(2)},...,V_{n-1}^{(2)}$ of $N^{(2)}$ are affinely
dependent. Then they must lie on a $(d-1)$-flat. We can then set
up our coordinate system so that hypersphere $(V_{i}^{(2)})$ has
the expression
\[
(x_{1}-c_{ix_{1}})^{2}+(x_{2}-c_{ix_{2}})^{2}+...+(x_{d-2}-c_{ix_{d-2}})^{2}+(x_{d-1}-c_{ix_{d-1}})^{2}+x_{d}^{2}=r_{i}^{2}.
\]

The intersection of any two hyperspheres $(V_{i}^{(2)})\cap(V_{j}^{(2)})$
contains $n-2$ points from $N^{(1)}$, distinct by hypothesis, so
that the hyperspheres are non-tangential. In addition, the intersection
lies on a hyperplane of the form 
\[
x_{1}(2c_{ix_{1}}-2c_{jx_{1}})+...+x_{d-1}(2c_{ix_{d-1}}-2c_{jx_{d-1}})=r_{j}^{2}-r_{i}^{2}+(c_{ix_{1}}^{2}-c_{jx_{1}}^{2})+...+(c_{ix_{d-1}}^{2}-c_{jx_{d-1}}^{2}).
\]

It easy to see then that the points $V_{n-2},V_{n-1},V_{n}\in\bigcap_{i=1}^{d-1}(V_{i}^{(2)})$
lie on a $2$-flat parallel to the $2$-flat on which $V_{1},V_{n-1},V_{n}\in\bigcap_{i=2}^{d}(V_{i}^{(2)})$
lie. Since the two planes intersect, they must be equal. Therefore
$V_{1},V_{n-2},V_{n-1},V_{n}$ are affinely dependent, a contradiction.
\end{proof}
From now on, we will tacitly assume that $N^{(1)}$ is nondegenerate
and not conhyperspherical. 

Our approach to proving the perspectivity $N^{(i)}$ and $N^{(i+2)}$
will naturally be through projective geometry. Specifically, we will
view Euclidean, Hyperbolic and Elliptic geometries as embedded inside
of real projective $n$-space. See the preliminaries section for an
overview of the relevant facts on projective geometry.

Let $\Gamma$ be a quadric in real projective $(n-2)$-space $\mathbb{R}P^{n-2}$.
Choose a polarity that fixes $\Gamma$. Let $H_{i}$ and $H_{i}'$
for $i=1,...,m$ be $m$ pairs of distinct hyperplanes and let $H_{i}\cap H_{i}'=h_{i},\mbox{ }i=1,...,m$.
In addition, let $H_{i}''$ be the polar of $h_{i}$ for each $i=1,...,m$.
We then have
\begin{lem}
\label{lem:quadric and polars}The $m$ $(n-4)-$flats $h_{1},...,h_{m}$
lie on a hyperplane if and only if the $m$ lines $H_{1}'',...,H_{m}''$
are concurrent.\end{lem}
\begin{proof}
Assume first that $h_{1},...,h_{m}$ lie on a hyperplane $L$. Then
$L$ is a conjugate hyperplane with respect to each $H_{i}''$. Therefore
the $H_{i}''$ all pass through the pole of $L$, which is a point.

Conversely, assume that the $H_{i}''$ are concurrent at a point $P$.
Then $P$ is conjugate to each $h_{i}$, so the $h_{i}$ all lie on
the polar of $P$, which is a hyperplane. 

\begin{figure}[h]
\includegraphics[scale=0.35]{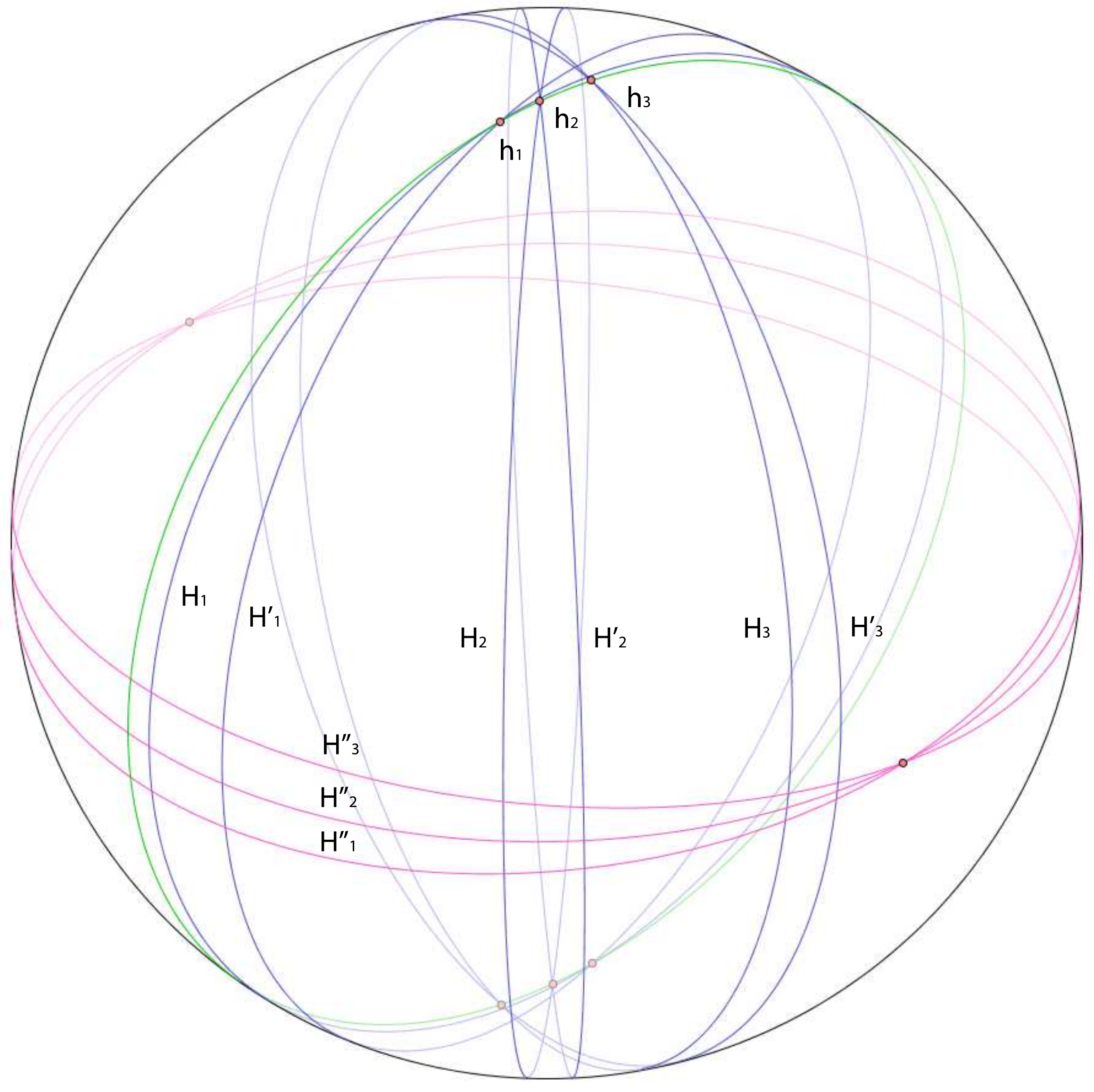}

\caption{Lemma \ref{lem:quadric and polars} on $\mathbb{S}^{2}$. Each $H_{i}''$
$(i=1,2,3)$ is the common perpendicular of $H_{i}$ and $H_{i}'$.
The points $h_{1},h_{2},h_{3}$ are collinear implying that $H_{1}'',H_{2}'',H_{3}''$
are concurrent, and conversely. }

\end{figure}

\end{proof}
The analogue of lemma \ref{lem:quadric and polars} in Euclidean geometry
that is of interest to us is the following trivial statement. As before,
we have $m$ pairs of hyperplanes $H_{i}$ and $H_{i}'$ for $i=1,...,m$.
Then $H_{i}$ and $H_{i}'$ are parallel if and only if some $m$
lines $H_{1}'',...,H_{m}''$, with each $H_{i}$ perpendicular to
both $H_{i}$ and $H_{i}'$, are concurrent. We are now ready to prove
the following Theorem:
\begin{thm}
\label{thm: Perspectivity in a point}In $(n-2)$-dimensional Euclidean,
Hyperbolic and Elliptic geometries, the sets of $n$ points $N^{(k)}$
and $N^{(k+2)}$ are perspective in a point.\end{thm}
\begin{proof}
Without loss of generality, assume that $k=1$. For simplicity, we
will denote $N^{(1)}$ by $N$, and the points $V_{i}^{(1)}$ similarly.
Let $N_{a,b}=\{V_{1},V_{2},...,V_{n}\}\setminus\{V_{a},V_{b}\}$ and
let $H_{a,b}$ be the supporting hyperplane of $N_{a,b}$. Define
$H_{a,b}^{(2)}$ similarly. By construction, line $V_{a}^{(1)}V_{b}^{(1)}$is
a common perpendicular to $H_{a,b}$ and $H_{a,b}^{(2)}$. As we vary
$b\in\{1,...,n\}\setminus\{a\}$, we obtain $n-1$ such lines all
concurrent at point $V_{a}^{(1)}$. By the converse of lemma \ref{lem:quadric and polars}
with $m=n-1$, the elements of the set $\{H_{a,b}\cap H_{a,b}^{(2)}\vert b\in\{1,...,n\}\setminus\{a\}\}$
lie on a hyperplane. Now consider the simplices $S_{a}=V_{a+1}\cdot\cdot\cdot V_{n+a-1}$
and $S_{a}^{(2)}=V_{a+1}^{(2)}\cdot\cdot\cdot V_{n+a-1}^{(2)}$. The
facets of $S_{a}$ and $S_{a}^{(2)}$ are $H_{a,b}$ and $H_{a,b}^{(2)}$
with $b\neq a$, respectively. We apply the generalized Desargues
theorem for $d$-dimensional space to the two simplices (see \cite{Bell}),
to conclude that they are perspective in a point. Call this point
$W$. By considering another pair of simplices, we conclude that they
too must be perspective in $W$, because the simplicies of the same
generation share parts, so that $N$ and $N^{(2)}$ are in perspective
about $W$.

\begin{figure}[h]
\includegraphics[scale=0.26]{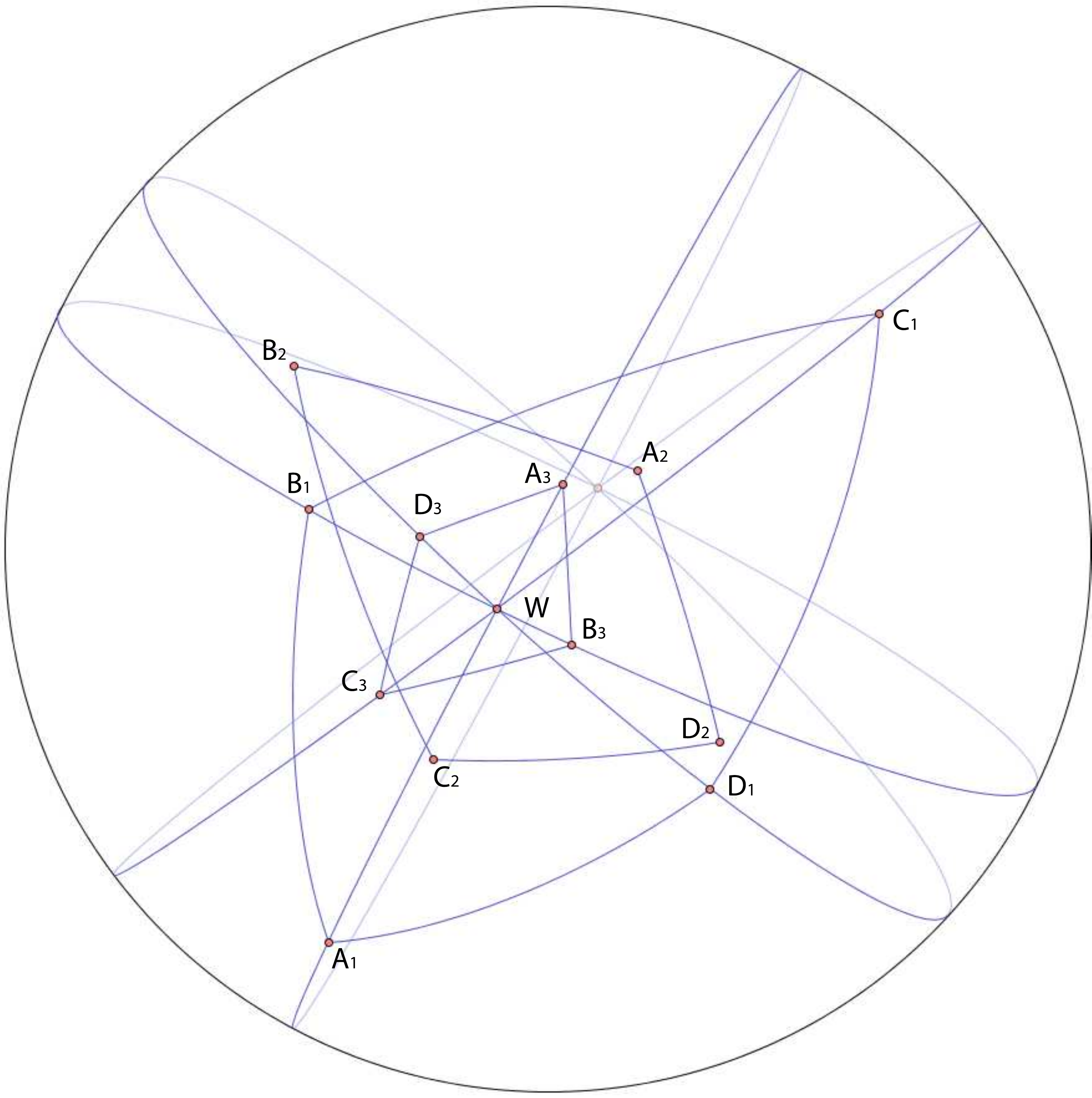}\includegraphics[scale=0.45]{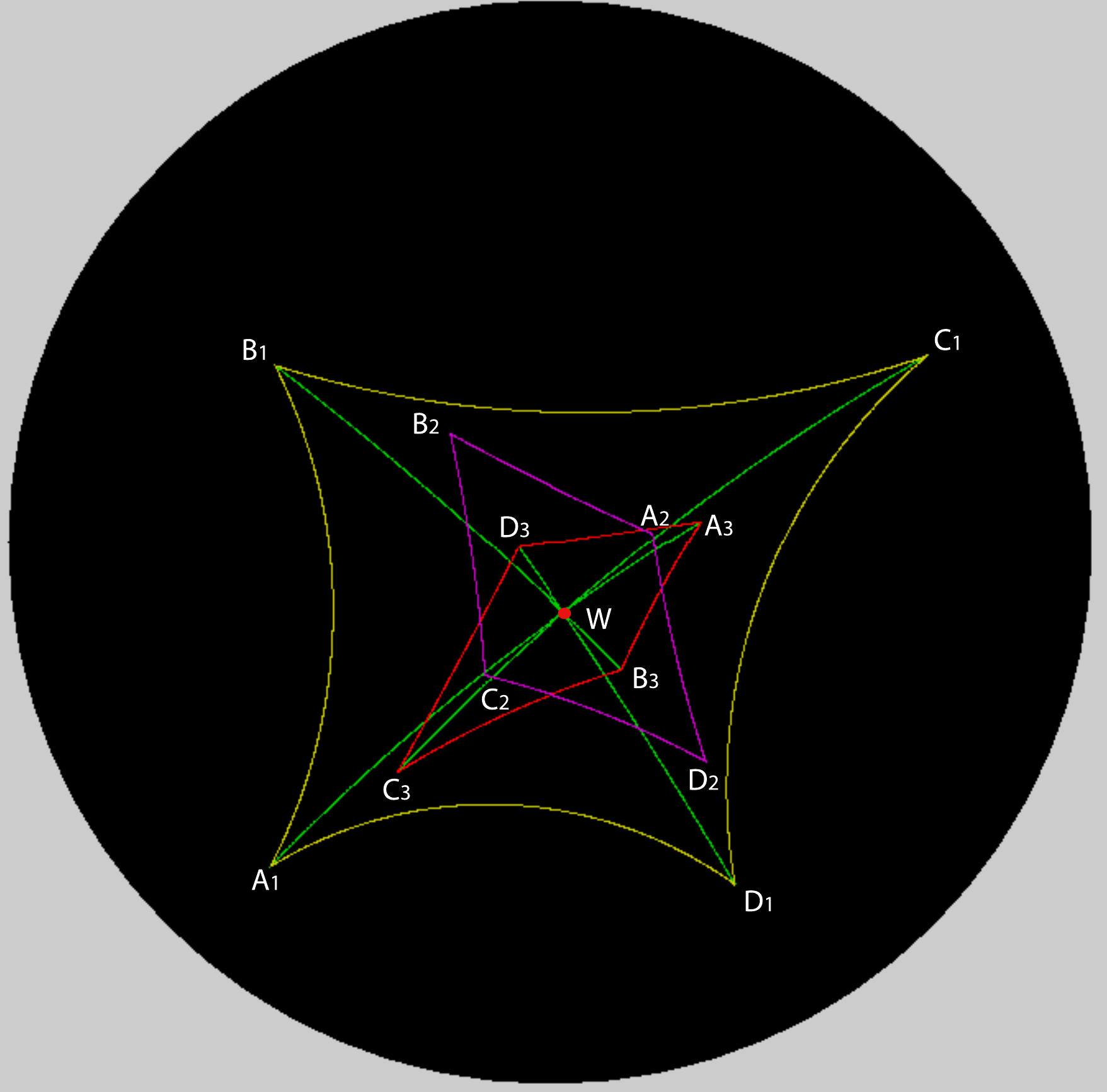}

\caption{Two special cases of Theorem \ref{thm: Perspectivity in a point}:
On the left is $\mathbb{S}^{2}$ and on the right is $2$-dimensional
Hyperbolic space viewed in the Poincar\'{e} disk model. The points
$A_{j},B_{j},C_{j}$ and $D_{j}$ are the members of the set $N^{(j)}$.
The point $W$ is the point about which $N^{(1)}$ and $N^{(3)}$
are in perspective. }

\end{figure}

\end{proof}
From the proof, it is not hard to see that
\begin{cor}
\label{thm:odd and even homothetic}In $(n-2)$-dimensional Euclidean
geometry, all sets of the form $N^{(2i+1)}$ are homothetic, and all
sets of the form $N^{(2k)}$ are homothetic.
\end{cor}
\begin{flushleft}
We will now show how to reverse the iterative construction, so that
given $N^{(i+1)}$ we can determine $N^{(i)}$. Recall that the isogonal
conjugate of a point $P$ with respect to a triangle $\triangle ABC$
in the plane is the point of intersection of the three lines obtained
by reflecting line $PA$ in the angle bisector of $\angle A$, line
$PB$ in the angle bisector of $\angle B$ and line $PC$ in the angle
bisector of $\angle C$. In case that $P$ lies on the circumcircle
of $\triangle ABC$, the isogonal conjugate is an ideal point. For
a more thorough discussion of isogonal conjugation in $\mathbb{R}^{2}$
and $\mathbb{R}^{3}$, we refer the reader to \cite{Grinberg} and
\cite{Court} respectively.
\par\end{flushleft}

\begin{flushleft}
For our purposes, we will not be using this definition of the isogonal
conjugate due to the ease and generality of the following definition,
which is equivalent to the former in $\mathbb{R}^{2}$ and $\mathbb{R}^{3}$:
\par\end{flushleft}
\begin{defn*}
\begin{flushleft}
Let $S=P_{1}\cdot\cdot\cdot P_{d+1}$ be a simplex in $d$-dimensional
space $V$, $P$ be a point not equal to $P_{1},...,P_{d+1}$ and
$P_{1}',...,P_{d+1}'$ be the $d+1$ reflections of $P$ in the facets
of $S$. Then the isogonal conjugate of $P$ with respect to $S$,
denoted by $Iso_{S}P=Iso_{P_{1}\cdot\cdot\cdot P_{d+1}}P$, is the
center of the hypersphere $(P_{1}'\cdot\cdot\cdot P_{d+1}')$. 
\par\end{flushleft}
\end{defn*}
\begin{flushleft}
The following property shows that isogonal conjugation is an involution.
\par\end{flushleft}
\begin{lem}
\label{lem:Iso Iso P =00003D P}With respect to any simplex $S$ in
Euclidean, Hyperbolic or Elliptic geometry, $Iso_{S}Iso_{S}P=P$.\end{lem}
\begin{proof}
Let $Q=Iso_{S}P$ and for $i=1,...,d+1$, let $P_{i}'$ and $Q_{i}'$
be the reflection of $P$, respectively $Q$, in the facet opposite
to $P_{i}$. We then have $PQ_{i}'=QP_{i}'$ for each $i=1,...,d+1$.
As $Q$ is the center of the hypersphere $(P_{1}'\cdot\cdot\cdot P_{d+1}')$,
we also have $QP_{i}'=QP_{j}'$ for every $i,j\in\{1,...,d+1\}$.
Since $PQ_{j}'=QP_{j}'$, $PQ_{j}'=QP_{i}'=PQ_{i}'$, so that $P$
is equidistant from all the $Q_{k}'$. Therefore $P$ is the center
of the hypersphere $(Q_{1}'Q_{2}'\cdot\cdot\cdot Q_{d+1}')$.
\end{proof}
Recall that we are using the notation that $V_{i}^{(2)}=V_{i+n}^{(2)}$.
The following Theorem allows us to reverse the iterative process:
\begin{thm}
\label{Thm:Reversing with Isogonal Conjugation}In $n$-dimensional
Euclidean, Hyperbolic and Elliptic geometry, $Iso_{V_{i+1}^{(2)}\cdot\cdot\cdot V_{n+i-1}^{(2)}}V_{i}^{(2)}=V_{i}$.\end{thm}
\begin{proof}
Consider the reflections of the vertex $V_{i}$ in each of the facets
of the simplex $V_{i+1}^{(2)}\cdot\cdot\cdot V_{n+i-1}^{(2)}$. Since
the facets of this simplex are the mediators $PB(V_{i}V_{j}),\forall j\in\{1,...,n\}\setminus\{i\}$,
reflecting $V_{i}$ in them results in the points $V_{j}$, $\forall j\in\{1,...,n\}\setminus\{i\}$.
The center of the hypersphere $(V_{i+1}\cdot\cdot\cdot V_{n+i-1})$
is by definition $V_{i}^{(2)}$, so that $Iso_{V_{i+1}^{(2)}\cdot\cdot\cdot V_{n+i-1}^{(2)}}V_{i}=V_{i}^{(2)}$.
By lemma \ref{lem:Iso Iso P =00003D P}, $Iso_{V_{i+1}^{(2)}\cdot\cdot\cdot V_{n+i-1}^{(2)}}V_{i}^{(2)}=V_{i}$.
\end{proof}
We will now shift our attention from the sets $N^{(i)}$ to the polytopes
$P^{(i)}$ with vertices $N^{(i)}$.
\begin{defn*}
Two polytopes $P$ and $P'$ are said to be \textit{combinatorially
equivalent }(or of the same \textit{combinatorial type}) provided
there exists a bijection $\phi$ between the set $\{F\}$ of all faces
of $P$ and the set $\{F'\}$ of all faces of $P'$, such that $F_{1}\subset F_{2}$
if and only if $\phi(F_{1})\subset\phi(F_{2})$ \cite{Grunbaum2}.
\end{defn*}
We will say that $P^{(i)}\sim P^{(j)}$ when a polytope $P^{(i)}$
with vertices $N^{(i)}$ can be chosen to have the same combinatorial
type as a polytope $P^{(j)}$ with vertices $N^{(j)}$, and $P^{(i)}$
and $P^{(j)}$ are similar.

Corollary \ref{thm:odd and even homothetic} then implies that in
$(n-2)$-dimensional Euclidean geometry,

\[
P^{(1)}\sim P^{(3)}\sim P^{(5)}\sim...
\]

and

\[
P^{(2)}\sim P^{(4)}\sim P^{(6)}\sim...
\]

Let $|P^{(i)}|$ denote the volume of $P^{(i)}$. From Corollary \ref{thm:odd and even homothetic}
it follows that for all $i,j$ and $k$, 
\[
\frac{|P^{(i)}|}{|P^{(i+2k)}|}=\frac{|P^{(j)}|}{|P^{(j+2k)}|}.
\]

In the case $d=2$, it is also true that $\frac{|P^{(i)}|}{|P^{(i+k)}|}=\frac{|P^{(j)}|}{|P^{(j+k)}|}$
. In fact, it is shown in \cite{RadkoTsukerman} that 
\[
\frac{|P^{(k+1)}|}{|P^{(k)}|}=\frac{1}{4}(\cot\alpha+\cot\gamma)\cdot(\cot\beta+\cot\delta),
\]
where $\alpha,\beta,\gamma$ and $\delta$ are the angles of the quadrilateral
$P^{(1)}$. However, experiment shows that the ratio of volumes of
consecutive polyheda is not in general only dependent on $P^{(1)}$.
Another property that holds for $d=2$, but not generally, is that
if $P^{(1)}$ is nondegenerate and noncyclic, then $P^{(2)}$ is never
cyclic. An easy way to see this is by applying isogonal conjugation
as in lemma \ref{Thm:Reversing with Isogonal Conjugation}, which
shows that $P^{(1)}$ must be at infinity. On the other hand, for
$d=3$, we can construct an example where $P^{(1)}$ is nondegenerate
and nonconhyperspherical and $P^{(2)}$ is conhyperspherical by using
the same lemma \ref{Thm:Reversing with Isogonal Conjugation} (see
figure \ref{fig:cyclic P^1}) because the isogonal conjugate of a
point on the circumsphere is not in general at infinity.

\begin{figure}[h]
\includegraphics[scale=0.35]{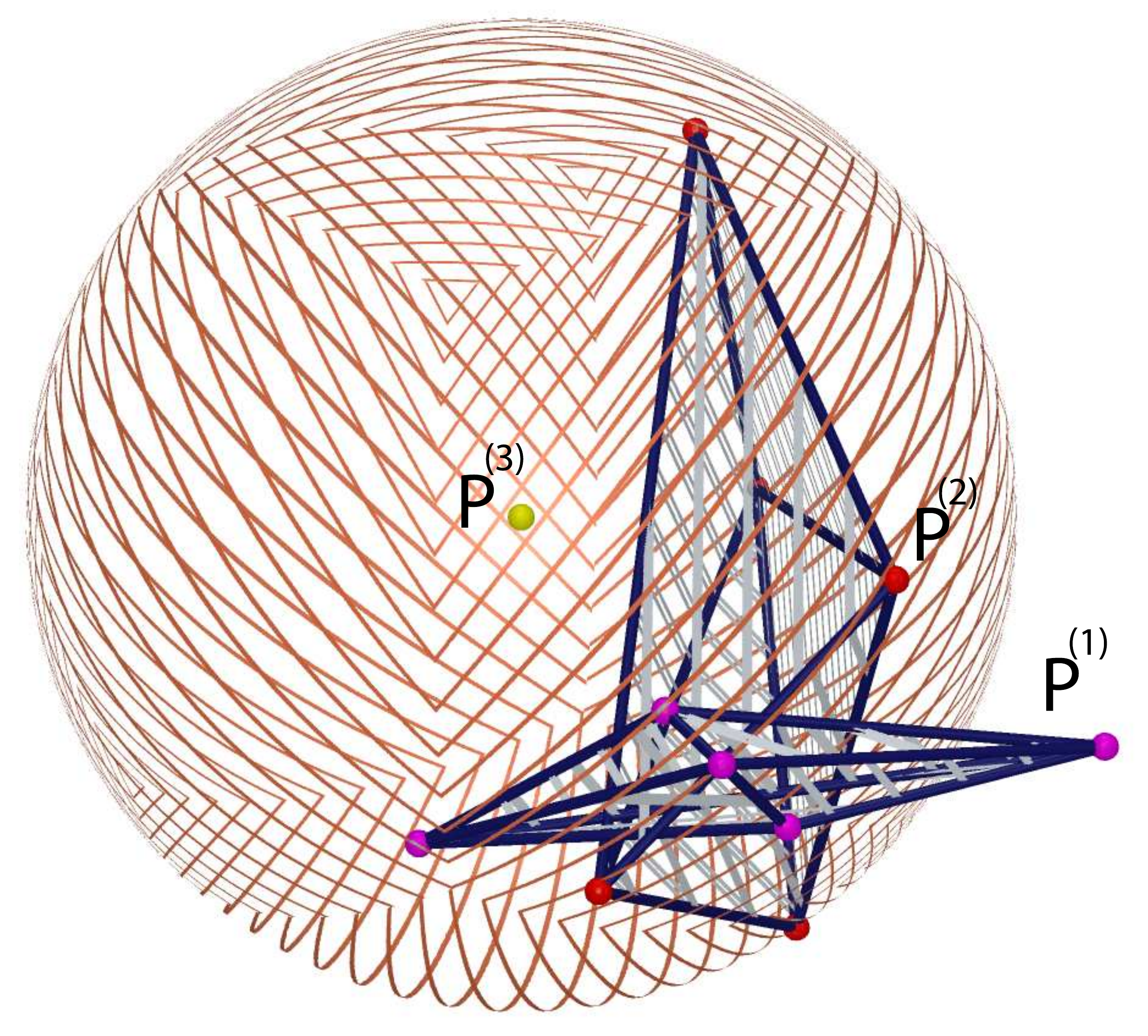}

\caption{\label{fig:cyclic P^1}The polyhedron $P^{(1)}$ is constructed to
be not conhyperspherical. The polyhedron $P^{(2)}$ obtained from
$P^{(1)}$ via the Perpendicular Bisector Construction, on the other
hand, is inscribed in a sphere. The next generation polyhedron $P^{(3)}$
is the center of the sphere. This phenomenon that $P^{(i)}$ is noncyclic
but $P^{(i+1)}$ is cyclic cannot occur in $\mathbb{R}^{2}$.}

\end{figure}

\section{The Isoptic Point in $\mathbb{R}^{d}$}

We now show that any pair of odd and any pair of even generation polytopes
are homothetic about the same point:
\begin{thm}
\label{thm:W is universal in R^n}The center of homothety $W^{(1)}$
of any pair of polytopes $P^{(2i+1)}$, $P^{(2j+1)}$ coincides with
the center of homothety $W^{(2)}$ of any pair of polytopes $P^{(2k)}$,
$P^{(2l)}$. \end{thm}
\begin{proof}
Let $M$ be the midpoint of segment $V_{a}V_{b}$ and $M^{(3)}$ that
of $V_{a}^{(3)}V_{b}^{(3)}$. For $c\notin\{a,b\}$, $V_{c}^{(2)}$
lies on the perpendicular bisector of $V_{a}V_{b}$, so that $V_{a}MV_{c}^{(2)}$
forms a right triangle. Similarly, $V_{a}^{(3)}M^{(3)}V_{c}^{(4)}$
is a right triangle. Since $P^{(1)}\sim P^{(3)}$, and $V_{c}$ ($V_{c}^{(3)})$
is the center of the hypersphere through all $V_{i}$ ($V_{i}^{(3)})$
, $i\in\{1,...,n\}\setminus\{c\}$, the two triangles are similar,
hence homothetic. Therefore $W^{(1)}=V_{b}V_{b}^{(3)}\cap M^{(1)}M^{(3)}\cap V_{c}^{(2)}V_{c}^{(4)}$.
Now consider $V_{d}^{(2)}$ and $V_{d}^{(4)}$ in place of $V_{c}^{(2)}$
and $V_{c}^{(4)}$ for some $d\notin\{a,b,c\}$. Then by the same
reasoning, $W^{(1)}=V_{b}V_{b}^{(3)}\cap M^{(1)}M^{(3)}\cap V_{d}^{(2)}V_{d}^{(4)}$.
But $W^{(2)}=V_{c}^{(2)}V_{c}^{(4)}\cap V_{d}^{(2)}V_{d}^{(4)}$.
Therefore $W^{(1)}=W^{(2)}$ .
\end{proof}
\begin{figure}[h]
\includegraphics[scale=0.3]{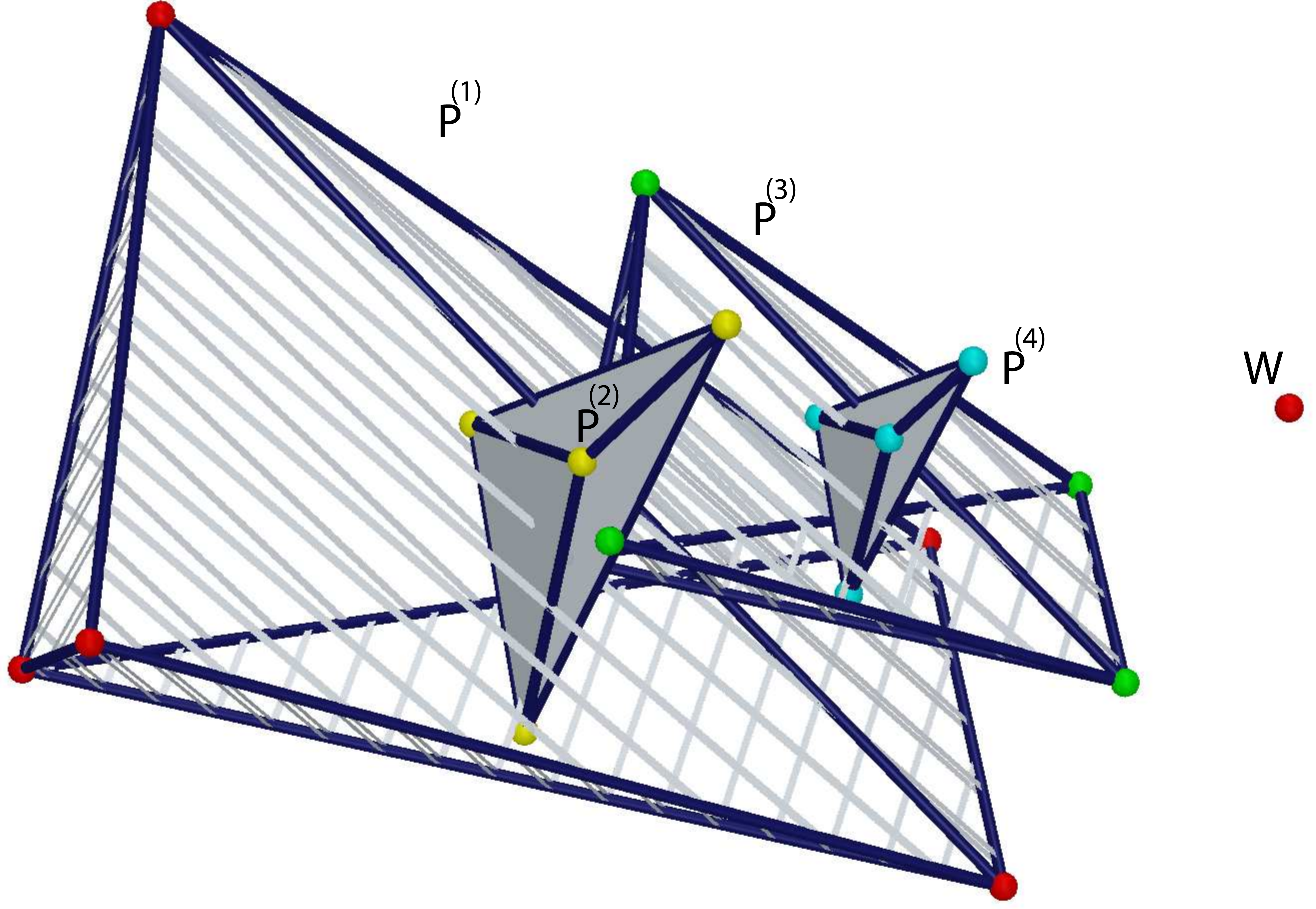}

\caption{An example illustrating Theorem \ref{thm:W is universal in R^n}:
The polyhedra $P^{(1)}$ and $P^{(3)}$ and the polyhedra $P^{(2)}$
and $P^{(4)}$ in $\mathbb{R}^{3}$ are homothetic about the same
point $W$.}
\end{figure}

We will call this {}``universal'' center of homothety $W$. This
point can be seen as the limit of the construction when $\frac{|P^{(1)}|}{|P^{(3)}|}>1$
and the limit of the reverse construction when $\frac{|P^{(1)}|}{|P^{(3)}|}<1$. 

In the case $d=2$, this point is called the \textit{Isoptic point
}due to its property of subtending equal angles at each triad circle
of the quadrilateral\textit{ }(see Radko and Tsukerman \cite{RadkoTsukerman})\textit{.
}In $\mathbb{R}^{2}$, $W$ has many properties that are analogous
to those of the circumcenter. More generally, if $N^{(1)}$ is approaching
a conhyperspherical configuration, then the limit of $W$ is the circumcenter
of $N^{(1)}$.

Finally, we pose the following problem. It is shown in \cite{RadkoTsukerman}
that in $\mathbb{R}^{2}$, the ratio of similarity of $P^{(i)}$ to
$P^{(i+2)}$ is equal to the following expressions: 
\[
\frac{1}{4}(\cot\alpha+\cot\gamma)\cdot(\cot\beta+\cot\delta)=\frac{1}{4}(\cot\alpha_{1}-\cot\beta_{2})\cdot(\cot\delta_{2}-\cot\gamma_{1})
\]
\[
=\frac{1}{4}(\cot\delta_{1}-\cot\alpha_{2})\cdot(\cot\beta_{1}-\cot\gamma_{2}),
\]

where the angles $\alpha_{i},\beta_{i},\gamma_{i},\delta_{i}$, $i=1,2,$
are the angles formed between sides and diagonals of a quadrilateral
(see figure \ref{fig:angles between sides diagonals}) and $\alpha=\alpha_{1}+\alpha_{2}$,
$\beta=\beta_{1}+\beta_{2}$, etc. Is there a similar expression for
the ratio of similarity in $\mathbb{R}^{n}$?

\begin{figure}[h]
\begin{centering}
\includegraphics[scale=0.14]{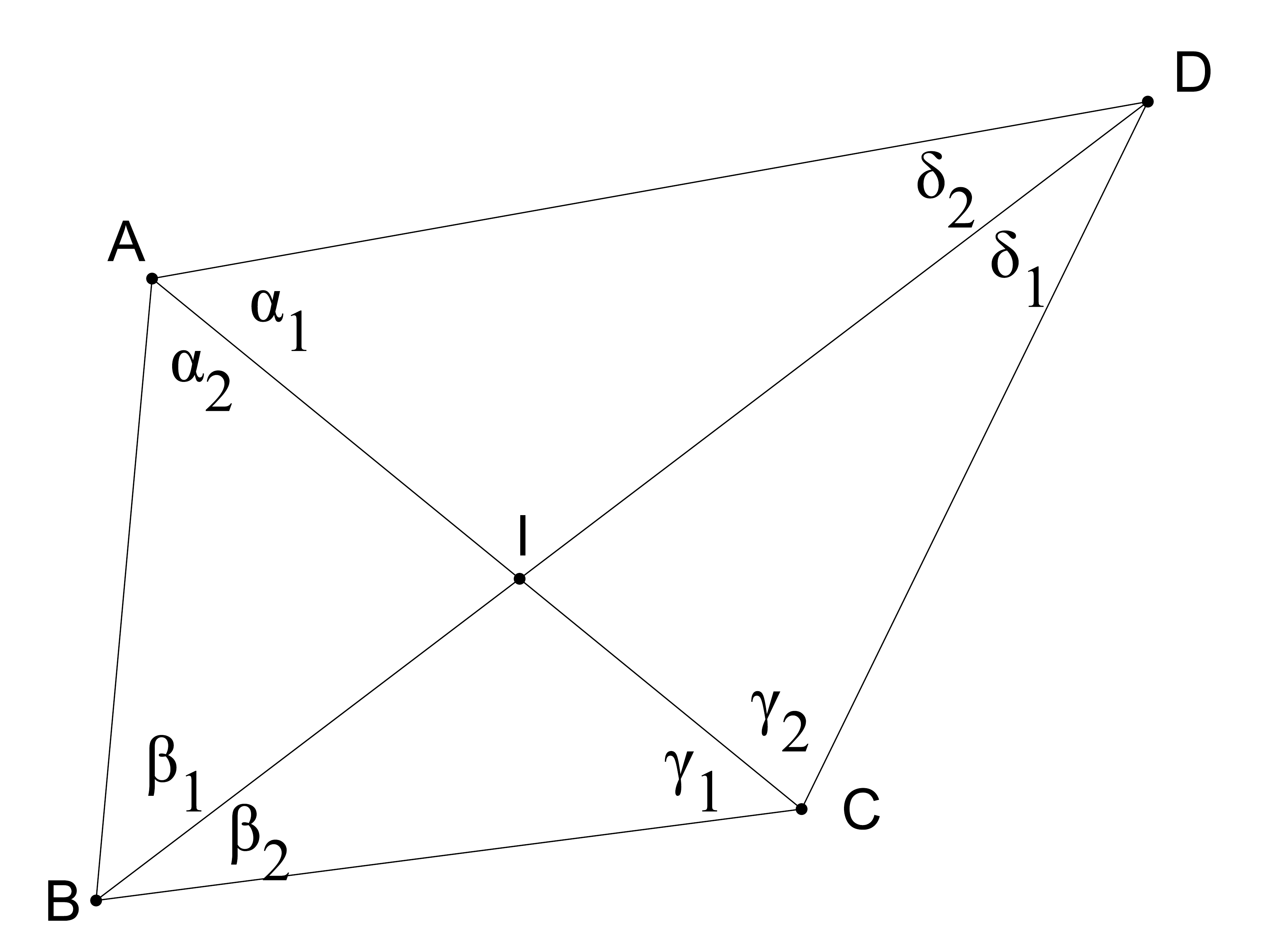}
\par\end{centering}

\caption{\label{fig:angles between sides diagonals}The angles between the
sides and diagonals of a quadrilateral.}
\end{figure}

$ $

Emmanuel Tsukerman: Stanford University

\textit{E-mail address: emantsuk@stanford.edu}

\begin{thebibliography}{References}
\bibitem{Langr}J. Langr, problem 1085E, \emph{Am. Math. Monthly,}
\textbf{60} (1953), 551. 

\bibitem{Ogilvy} C.S. Ogilvy, \emph{Tomorrow's Math. Unsolved Problems
for the Amateur, }Oxford University Press, New York, 1962. (2nd edn.,
1972. Pages reference second edition). 

\bibitem{cut-the-knot}A. Bogomolny, \emph{Quadrilaterals Formed by
Perpendicular Bisectors,} Interactive Mathematics Miscellany and Puzzles,
http://www.cut-the-knot.org/Curriculum/Geometry/PerpBisectQuadri.shtml,
Accessed 06 December 2010. 

\bibitem{Grunbaum}B. Grünbaum, \emph{Quadrangles, Pentagons, and
Computers}, Geombinatorics 3 (1993), 4-9.

\bibitem{D. Bennet}D. Bennett, \emph{Dynamic Geometry Renews Interest
in an Old Problem}, Geometry Turned On, MAA Notes 41, 1997, 25-28.

\bibitem{King}J. King, \emph{Quadrilaterals Formed by Perpendicular
Bisectors}, in Geometry Turned On, MAA Notes 41, 1997, 29--32. 

\bibitem{Shepard}G.C. Shephard, \emph{The perpendicular bisector
construction'', }Geom. Dedicata, \textbf{56}, 75-84 (1995). 

\bibitem{RadkoTsukerman} O. Radko and E. Tsukerman, \textit{{}``The
Perpendicular Bisectors Construction, the Isoptic Point and the Simson
Line of a Quadrilateral'', }to appear in \textit{Forum Geometricorum.}

\bibitem{Mathworld}Weisstein, Eric W. \textquotedbl{}Polar.\textquotedbl{}
From MathWorld--A Wolfram Web Resource. http://mathworld.wolfram.com/Polar.html

\bibitem{Cederberg}J. N. Cederberg, A Course in Modern Geometries,
2nd edition, p. 283-298.

\bibitem{Bell}P. O. Bell, Generalized Theorems of Desargues for $n$-Dimensional
Projective Space, Proceedings of the American Mathematical Society
Vol. 6, No. 5 (Oct., 1955), pp. 675-681.

\bibitem{Grinberg}D. Grinberg, Isogonal conjugation with respect
to triangle, unpublished notes, http://www.cip.i.lmu.de/\textasciitilde{}grinberg/Isogonal.zip,
last accessed on 1/24/2010.

\bibitem{Court} N. A. Court, Modern Pure Solid Geometry. The Macmillan
Co., 1935.

\bibitem{Grunbaum2}B. Gr\"unbaum, Convex Polytopes, 2nd edition,
p. 38.

\end{thebibliography}
\end{document}